\documentclass[a4paper,10pt]{article}
\usepackage[latin1]{inputenc}
\usepackage[T1]{fontenc}
\usepackage{amsmath}
\usepackage{amsfonts}
\usepackage{amstext}
\usepackage{amsthm}
\usepackage[all]{xy}
\usepackage{geometry}
\usepackage{srcltx}
\usepackage{graphics}
\usepackage{epsfig}
\usepackage{subfigure}
\usepackage{slashed}
\usepackage{color}
\usepackage{amssymb}
\usepackage{srcltx}
\newtheorem{theorem}{Theorem}[section]
\newtheorem{lemma}[theorem]{Lemma}

\newtheorem{prop}[theorem]{Proposition}
\newtheorem{cor}[theorem]{Corollary}

\DeclareMathOperator{\im}{im}

\DeclareMathOperator{\sfl}{sf}
\DeclareMathOperator{\diag}{diag}

\title{The Equivariant Spectral Flow and Bifurcation of Periodic Solutions of Hamiltonian Systems}
 
\author{Marek Izydorek, Joanna Janczewska and Nils Waterstraat}

\begin{document}
\date{}
\maketitle


\begin{abstract}
We define a spectral flow for paths of selfadjoint Fredholm operators that are equivariant under the orthogonal action of a compact Lie group as an element of the representation ring of the latter. This $G$-equivariant spectral flow shares all common properties of the integer valued classical spectral flow, and it can be non-trivial even if the classical spectral flow vanishes. Our main theorem uses the $G$-equivariant spectral flow to study bifurcation of periodic solutions for autonomous Hamiltonian systems with symmetries. 
\end{abstract}

\section{Introduction}
The spectral flow is a homotopy invariant for paths of selfadjoint Fredholm operators that was invented by Atiyah, Patodi and Singer in their study of spectral asymmetry and index
theory in \cite{AtiyahPatodi}. Selfadjoint Fredholm operators are either invertible or $0$ is an isolated eigenvalue of finite multiplicity. Roughly speaking, if $\mathcal{A}=\{\mathcal{A}_\lambda\}_{\lambda\in[0,1]}$ is a path of selfadjoint Fredholm
operators, then the spectral flow of $\mathcal{A}$ is the net number of eigenvalues of $\mathcal{A}_0$ that become positive whilst the parameter $\lambda$ travels along the unit interval. The spectral flow has been widely used in different communities over the last decades, e.g., in global analysis, mathematical physics, symplectic analysis and bifurcation theory. Though we are particularly interested in the latter two, the first part of this work is dealing with the spectral flow as an abstract mathematical object and thus shall be of general interest.\\
Atiyah, Patodi and Singer's original approach to the spectral flow in \cite{AtiyahPatodi} involves studying the graph of the spectrum of the path after a suitable perturbation and counting intersection numbers with the horizontal axis. Robbin and Salamon showed in \cite{Robbin-Salamon} that these intersection numbers can actually be computed under generic assumptions by signatures of quadratic forms, which yields an independent construction of the spectral flow. A description of the spectral flow without intersection numbers was introduced by Floer in \cite{Floer}, and Phillips gave in \cite{Phillips} a rigorous elaboration of this approach in purely functional analytic terms.\\
Let us now assume for a moment that $H$ is a complex Hilbert space and $G$ a compact Lie group acting unitarily on $H$. We denote by $R(G)$ the complex representation ring of $G$, which consists of all formal differences of isomorphism classes of finite dimensional linear complex representations of $G$. Fang invented in \cite{Fang} an equivariant spectral flow for studying equivariant versions of theorems in global analysis that originally appeared in Atiyah, Patodi and Singer's work \cite{AtiyahPatodi}. In a slightly more functional analytic language, this construction can be outlined as follows. If $\mathcal{A}=\{\mathcal{A}_\lambda\}_{\lambda\in [0,1]}$ is a path of unbounded selfadjoint Fredholm operators with compact resolvent on $H$, then the spectrum of each $\mathcal{A}_\lambda$ is discrete and each eigenvalue is of finite multiplicity. If the operators $\mathcal{A}_\lambda$ are $G$-equivariant, then for each $\mu\in\sigma(\mathcal{A}_\lambda)$, the eigenspace $E(\mathcal{A}_\lambda,\mu)$ is invariant under $G$ and consequently a finite dimensional complex representation of $G$. Fang showed that there are at most countable many continuous functions $f_j$ and elements $R_j\in R(G)$ such that

\[\{(\mu,E(\mathcal{A}_\lambda,\mu)):\, \mu\in\sigma(\mathcal{A}_\lambda)\}=\bigcup_{j\in\mathbb{N}}{\{(f_j(\lambda),R_j)\}}\subset \mathbb{R}\times R(G),\quad \lambda\in I.\]
He defined the equivariant spectral flow of the path $\mathcal{A}$ by

\begin{align}\label{Fang}
\sfl_G(\mathcal{A})=\sum_{j\in\mathbb{N}}{\varepsilon(f_j)R_j}\in R(G),
\end{align}
where $\varepsilon(f_j)$ is the intersection number of the graph of $f_j$ with the line $\lambda=-\delta$ for a sufficiently small $\delta>0$. As only finitely many of the intersection numbers are non-zero, the right hand side in the definition is actually a finite sum. If $G$ is the trivial group, then $R(G)\cong\mathbb{Z}$ and \eqref{Fang} can be identified with the sum of the intersection numbers. Then \eqref{Fang} is the classical definition of the spectral flow of Atiyah, Patodi and Singer \cite{AtiyahPatodi}, who also were only dealing with operators having compact resolvents as this is a common setting when studying elliptic operators on closed manifolds.\\
The first aim of this paper is to define the $G$-equivariant spectral flow as element of the representation ring in purely functional analytic terms as in Phillips approach \cite{Phillips} in the classical case. We show that all common properties of the spectral flow carry over to this setting, and point out that the $G$-equivariant spectral flow is a finer invariant than the classical one, i.e., a trivial $G$-equivariant spectral flow implies that the classical spectral flow is trivial as well. We underpin this latter observation by a simple example of a class of paths of $\mathbb{Z}_2$-equivariant selfadjoint Fredholm operators having a vanishing spectral flow but a non-trivial $\mathbb{Z}_2$-equivariant spectral flow.\\
The second aim of this paper concerns bifurcation theory for critical points of one-parameter families of functionals. The Hessians of such functionals are selfadjoint operators and they play a crucial role in finding bifurcation points. Actually, under the additional assumption that the operators are Fredholm, a non-trivial kernel is a necessary assumption for the existence of a bifurcation. It is well known that a jump in the Morse index of the Hessians causes a bifurcation if the functionals are essentially positive, i.e. their Morse indices are finite (cf. \cite{Mawhin}). However, there are important types of differential equations whose solutions are critical points of a functional on a Hilbert space but where the Hessians fail to have finite Morse indices. The most prominent class of this type are Hamiltonian systems and numerous works have appeared over the last decades where approaches were developed to deal with their bifurcation problems (cf., e.g., \cite{Mawhin}, \cite{Szulkin}, \cite{BartschSzulkin}, \cite{SFLPejsachowiczII}). We want to emphasize that the Hessians are selfadjoint Fredholm operators in this case and the latter reference uses the main theorem of \cite{SFLPejsachowiczI}, which shows the existence of a bifurcation of critical points if the spectral flow of the path of Hessians is non-trivial.\\ 
Smoller and Wasserman considered in \cite{SmollerWasserman} essentially positive functionals that are invariant under the orthogonal action of a compact Lie group. Note that the Hessians of such functionals are $G$-equivariant. They introduced a bifurcation invariant in terms of representations of $G$ and applied it to symmetry breaking bifurcation for semilinear elliptic equations. We firstly show below that, for essentially positive functionals, Smoller and Wasserman's bifurcation invariant actually is our $G$-equivariant spectral flow of the path of Hessians. Secondly, we use the latter observation and the fact that the $G$-equivariant spectral flow is defined for any path of $G$-equivariant selfadjoint Fredholm operators, to open up the methods from \cite{SmollerWasserman} to strongly indefinite functionals. Indeed, in our first main theorem, we compute the $G$-equivariant spectral flow of paths of $G$-equivariant autonomous Hamiltonian systems in terms of representations of the group $G$. Secondly, we show that a non-trivial $G$-equivariant spectral flow yields bifurcation in this setting. If the group action is trivial, i.e., when the $G$-equivariant spectral flow is just the ordinary spectral flow, we obtain as corollary the main theorem of \cite{SFLPejsachowiczII}.\\  
Rybicki and his group have developed a degree theory for strongly indefinite $G$-invariant functionals that can show the existence of (global) bifurcation points for some kinds of equations (see, e.g., \cite{GawryRy} and \cite{GoleRy}). Rybicki's bifurcation invariant is an element of the tom-Dieck ring $U(G)$, which is made of equivalence classes of $G$-homotopy types of finite $G$-$CW$-complexes. The $G$-equivariant spectral flow is different from Rybicki's bifurcation invariant and it seems to us more comprehensive and easier to compute, which we demonstrate by our main theorems.\\
We conclude the introduction by pointing out a possible application of our $G$-equivariant spectral flow in global analysis. A non-vanishing (classical) spectral flow implies the existence of non-trivial kernels in a path of selfadjoint Fredholm operators. B\"ar showed in his celebrated work \cite{BaerHarmonic} the existence of harmonic spinors on closed spin manifolds of dimension $n\equiv 3\mod 4$ by constructing paths of Dirac operators having non-trivial spectral flows (see also \cite{BaerSpinors},\cite{Spinors}). This argument does not carry over to the remaining dimensions as in these cases symmetries of the Dirac operators force the spectral flow to vanish \cite[\S 8]{BaerHarmonic}. We hope that the equivariant spectral flow can contribute to this open problem as its non-triviality implies the existence of non-trivial kernels even if the classical spectral flow vanishes.\\     
Henceforth, as we are mainly interested in bifurcation theory, we consider only real Hilbert spaces and thus the $G$-equivariant spectral flow is an element of the real representation ring $RO(G)$ of the acting Lie group $G$. All of our arguments carry over mutatis mutandis to complex Hilbert spaces, where the $G$-equivariant spectral flow is an element of the complex representation ring $R(G)$. Moreover, if we consider paths of unbounded operators, we always assume a common domain. The interested reader will have no difficulties to work out the obvious modifications for paths of $G$-equivariant unbounded selfadjoint Fredholm operators that are continuous in the gap-topology as in \cite{UnbSpecFlow}.


\section{The $G$-equivariant Spectral Flow}
 
\subsection{Recap: The Spectral Flow}
The aim of this section is to recall the construction of the classical spectral flow, where we follow \cite{Phillips} and use the terminology from \cite{CalcVar}.\\
Let $W$ and $H$ be real separable Hilbert spaces with a dense embedding $\iota:W\hookrightarrow H$. We denote by $\mathcal{L}(W,H)$ the space of all linear bounded operators with the operator norm. As $W\subset H$ is dense, it makes sense to define $\mathcal{S}(W,H)$ as the set of all operators in $\mathcal{L}(W,H)$ which are selfadjoint when considered as operators on $H$ having the dense domain $W$. Note that operators in $\mathcal{S}(W,H)$ are bounded as operators between the Hilbert spaces $W$ and $H$. Henceforth, we consider $\mathcal{S}(W,H)$ as metric space with respect to the metric inherited from $\mathcal{L}(W,H)$. Finally, we denote by $\mathcal{FS}(W,H)$ the Fredholm operators in $\mathcal{S}(W,H)$, i.e., the operators $T\in\mathcal{S}(W,H)$ such that $\ker(T)$ is of finite dimension and $\im(T)$ is closed.\\ 
The spectrum $\sigma(T)$ of $T\in\mathcal{FS}(W,H)$ is the (generally non-disjoint) union of the point spectrum $\sigma_p(T)$ and the essential spectrum $\sigma_{ess}(T)$. If $0\in\sigma(T)$, then it is an isolated point of $\sigma(T)$ and an eigenvalue of finite multiplicity (see, e.g., \cite[Lemma 13]{Fredholm}). We denote for $a,b\notin\sigma(T)$ by $\chi_{[a,b]}(T)$ the spectral projection of $T$ with respect to the interval $[a,b]$. Note that if $\sigma_{ess}(T)\cap{[a,b]}=\emptyset$, then $\chi_{[a,b]}(T)$ is the orthogonal projection onto the direct sum of the eigenspaces for eigenvalues in $[a,b]$. In this case $\chi_{[a,b]}(T)$ is of finite rank as there are only isolated eigenvalues of finite multiplicity in $[a,b]$. The following lemma is a well known result about the stability of spectra (cf. \cite[I.II]{Gohberg}).

\begin{lemma}\label{lemma-sflconstruction}
Let $T_0\in\mathcal{FS}(W,H)$ and $a>0$ such that $\pm a\notin\sigma(T_0)$ and $[-a,a]\cap\sigma_{ess}(T_0)=\emptyset$. Then there is an open neighbourhood $N_{T_0,a}$ of $T_0$ in $\mathcal{FS}(W,H)$ such that $\pm a\notin\sigma(T)$ and $[-a,a]\cap\sigma_{ess}(T)=\emptyset$ for all $T\in N_{T_0,a}$. 
\end{lemma}
\noindent
Let us note that if $N_{T_0,a}$ denotes a neighbourhood as in the previous lemma, then the map

\[N_{T_0,a}\ni T\mapsto\chi_{[-a,a]}(T)\in\mathcal{L}(H)\]
is continuous and of constant (finite) rank on each connected component of $N_{T_0,a}$.\\
The construction of the spectral flow of a path $\mathcal{A}=\{\mathcal{A}_\lambda\}_{\lambda\in I}$ is now as follows, where we denote by $I$ the unit interval. For every $\lambda\in I$ there is an open neighbourhood $N_{\lambda,a}\subset\mathcal{FS}(W,H)$ of $\mathcal{A}_\lambda$ as in the previous lemma. The preimages of these neighbourhoods define an open cover of the compact interval $I$. Consequently, there is a partition of the interval $0=\lambda_0<\lambda_1<\ldots<\lambda_N=1$ and numbers $a_i>0$ such that 

\[\pm a_i\notin\sigma(\mathcal{A}_\lambda),\quad [-a_i,a_i]\cap\sigma_{ess}(\mathcal{A}_\lambda)=\emptyset,\qquad \lambda\in[\lambda_{i-1},\lambda_i].\] 
In particular, the maps

\[[\lambda_{i-1},\lambda_i]\ni \lambda\mapsto\chi_{[-a_i,a_i]}(\mathcal{A}_\lambda)\in\mathcal{L}(H)\]
are continuous for $i=1,\ldots,N$. The \textit{Spectral Flow} of the path $\mathcal{A}$ is the integer

\begin{align}\label{sfl}
\sfl(\mathcal{A})=\sum^N_{i=1}{(\dim E(\mathcal{A}_{\lambda_i},[0,a_i])-\dim E(\mathcal{A}_{\lambda_{i-1}},[0,a_i]))},
\end{align}  
where we denote by $E(\mathcal{A}_\lambda,[a,b])$ for $[a,b]\cap\sigma_{ess}(\mathcal{A}_\lambda)=\emptyset$ the direct sum of the eigenspaces for eigenvalues in $[a,b]$.\\ 
It was shown in \cite{Phillips} that this definition neither depends on the choice of the partition $0=\lambda_0<\lambda_1<\ldots<\lambda_N=1$ of the unit interval nor on the numbers $a_i>0$, $i=1,\ldots,N$. Moreover, the spectral flow can be uniquely characterised by some of its properties (see, e.g., \cite{Lesch}, \cite{JacoboUniqueness}, \cite{CompSfl}), and among them is its quite remarkable homotopy invariance. We do not recall here any of these properties, as most of them will follow as special cases of our equivariant spectral flow, which we discuss in the next section.


\subsection{The $G$-equivariant Spectral Flow: Definition}
Let $G$ be a compact Lie group. We denote by $RO(G)$ the real representation ring of $G$, i.e., the Grothendieck group of all isomorphism classes of finite dimensional linear real representations of $G$ (see \cite{Segal}). The elements of $RO(G)$ are formal differences $[U]-[V]$ of isomorphism classes of $G$-representations modulo the equivalence relation generated by $[U]-[V]\sim [U\oplus W]-[V\oplus W]$. The neutral element in $RO(G)$ is $[V]-[V]$ for any $G$-representation $V$ and the inverse element of $[U]-[V]$ is $[V]-[U]$. Henceforth we use the common name representation ring even though we will never use the ring structure of $RO(G)$ and consider it merely as an abelian group.\\ 
Let now $H$ be a real separable Hilbert space on which $G$ acts orthogonally and such that $W\subset H$ is invariant under this action. We consider paths $\mathcal{A}=\{\mathcal{A}_\lambda\}_{\lambda\in I}$ of selfadjoint Fredholm operators in $\mathcal{FS}(W,H)$ that are $G$-equivariant, i.e., $\mathcal{A}_\lambda(gu)=g(\mathcal{A}_\lambda u)$ for all $u\in W$ and $g\in G$. As in \eqref{sfl}, we let $0=\lambda_0<\ldots< \lambda_N=1$ be a partition of the unit interval and $a_i>0$, $i=1,\ldots N$, such that $[\lambda_{i-1},\lambda_i]\ni \lambda\mapsto\chi_{[-a_i,a_i]}(\mathcal{A}_\lambda)\in\mathcal{L}(H)$ are continuous families of finite rank projections. Then, for $i=1,\ldots,N$, the spaces $E(\mathcal{A}_\lambda,[0,a_i])$, $\lambda_{i-1}\leq \lambda\leq \lambda_i$, in \eqref{sfl} are finite dimensional real $G$-representations. We define the \textit{$G$-equivariant spectral flow} of $\mathcal{A}$ by   

\begin{align}\label{sfl-equiv}
\sfl_G(\mathcal{A})=\sum^N_{i=1}{([E(\mathcal{A}_{\lambda_i},[0,a_i])]-[E(\mathcal{A}_{\lambda_{i-1}},[0,a_i])])}\in RO(G).
\end{align} 
Note that if $G$ is trivial, then representations are isomorphic if and only if they are of the same dimension. Hence $RO(G)\cong\mathbb{Z}$ in this case and so \eqref{sfl-equiv} can be identified with the ordinary spectral flow \eqref{sfl}. To show that \eqref{sfl-equiv} is well defined, we need the following lemma that will be used below several times.

\begin{lemma}\label{lemma-technical}
Let $[c,d]\subset[0,1]$ and $[a,b]\subset\mathbb{R}$ such that $a,b\notin\sigma(\mathcal{A}_\lambda)$ and $[a,b]\cap\sigma_{ess}(\mathcal{A}_\lambda)=\emptyset$ for $\lambda\in[c,d]$. Then $E(\mathcal{A}_c,[a,b])$ and $E(\mathcal{A}_d,[a,b])$ are isomorphic $G$-representations.  
\end{lemma} 

\begin{proof}
The map $[c,d]\ni\lambda\mapsto\chi_{[a,b]}(\mathcal{A}_{\lambda})\in\mathcal{L}(H)$ is continuous under the given assumptions. Hence there is a partition $c=\lambda_0\leq \lambda_1\leq\ldots\leq \lambda_k=d$ such that $\|\chi_{[a,b]}(\mathcal{A}_{\lambda_i})-\chi_{[a,b]}(\mathcal{A}_{\lambda_{i-1}})\|<1$. Moreover, these projections are $G$-equivariant as their images are invariant and $G$ acts orthogonally. Consequently, all we need to show is that if $Q$ and $P$ are $G$-equivariant finite rank projections such that $\|P-Q\|<1$, then $\im(P)$ and $\im(Q)$ are isomorphic as $G$-representations. We first note that $\im(P)$ and $\im(Q)$ are of the same finite dimension (cf. \cite[Lem. I.4.3]{Gohberg}). The $G$-equivariant map $U:=PQ+(I_H-P)(I_H-Q)$ maps $\im(P)$ into $\im(Q)$. A direct computation shows that

\[(QP+(I_H-Q)(I_H-P))U=I_H-(P-Q)^2.\]
As $\|P-Q\|<1$, the right hand side is an isomorphism which shows that $U$ is injective. Thus $U\mid_{\im(P)}:\im(P)\rightarrow\im(Q)$ is a $G$-equivariant isomorphism. 
\end{proof}
\noindent
A first application of the previous lemma is the well-definedness of \eqref{sfl-equiv}.

\begin{lemma}
The $G$-equivariant spectral flow is well defined, i.e., \eqref{sfl-equiv} does not depend on the choice of the partition $0=\lambda_0< \lambda_1<\ldots< \lambda_N=1$ of $I$ and the numbers $a_i>0$.
\end{lemma} 

\begin{proof}
Our argument follows \cite{Phillips} and consists of three steps. Firstly, we add a $\lambda_\ast$ to the partition $0=\lambda_0< \lambda_1<\ldots< \lambda_N=1$, say $\lambda_i<\lambda_\ast<\lambda_{i+1}$. Then the only amendment in \eqref{sfl-equiv} is that 

\[[E(\mathcal{A}_{\lambda_i},[0,a_i])]-[E(\mathcal{A}_{\lambda_{i-1}},[0,a_i])]\]
is replaced by

\begin{align*}
([E(\mathcal{A}_{\lambda_i},[0,a_i])]-[E(\mathcal{A}_{\lambda_\ast},[0,a_i])])+([E(\mathcal{A}_{\lambda_\ast},[0,a_i])]-[E(\mathcal{A}_{\lambda_{i-1}},[0,a_i])]).
\end{align*}
As

\[[E(\mathcal{A}_{\lambda_\ast},[0,a_i])]-[E(\mathcal{A}_{\lambda_\ast},[0,a_i])]=0\in RO(G),\]
this does not affect \eqref{sfl-equiv}.\\
Secondly, let us consider the case that we have $b_i\neq a_i$ for some $i$ such that $\pm b_i\notin\sigma(\mathcal{A}_\lambda)$ and $[-b_i,b_i]\cap\sigma_{ess}(\mathcal{A}_\lambda)=\emptyset$ for $\lambda\in[\lambda_{i-1},\lambda_i]$. We assume without loss of generality that $b_i>a_i$. As $E(\mathcal{A}_{\lambda_i},[0,a_i])$ and $E(\mathcal{A}_{\lambda_i},[a_i,b_i])$ are $G$-invariant and intersect trivially, we obtain

\begin{align*}
&[E(\mathcal{A}_{\lambda_i},[0,b_i])]-[E(\mathcal{A}_{\lambda_{i-1}},[0,b_i])]\\
&=([E(\mathcal{A}_{\lambda_i},[0,a_i])\oplus E(\mathcal{A}_{\lambda_i},[a_i,b_i])])-([E(\mathcal{A}_{\lambda_{i-1}},[0,a_i])\oplus E(\mathcal{A}_{\lambda_{i-1}},[a_i,b_i])])\\
&=([E(\mathcal{A}_{\lambda_i},[0,a_i])]-[E(\mathcal{A}_{\lambda_{i-1}},[0,a_i])])+([E(\mathcal{A}_{\lambda_i},[a_i,b_i])]-[E(\mathcal{A}_{\lambda_{i-1}},[a_i,b_i])])\\
&=[E(\mathcal{A}_{\lambda_i},[0,a_i])]-[E(\mathcal{A}_{\lambda_{i-1}},[0,a_i])],
\end{align*} 
where the last equality follows from Lemma \ref{lemma-technical}. Thus \eqref{sfl-equiv} does not depend on the choice of the $a_i$.\\
Finally, if we have two different partitions $0=\lambda_0< \lambda_1<\ldots< \lambda_N=1$ and $0=\eta_0< \eta_1<\ldots< \eta_M=1$ with corresponding numbers $a_i$, $i=1,\ldots, N$ and $b_i$, $i=1,\ldots, M$, we merge the two partitions to obtain a finer one for which we can use either the $a_i$ or the $b_i$ to compute \eqref{sfl-equiv}. In both cases this does not affect \eqref{sfl-equiv} by the first step of our proof. As \eqref{sfl-equiv} does not depend on the $a_i$ and $b_i$ by the second step, we finally see that we obtain in both cases the same element in $RO(G)$.
\end{proof}
\noindent
We conclude this section by noting that there is a canonical homomorphism 

\[F:RO(G)\rightarrow\mathbb{Z},\quad [U]-[V]\mapsto\dim(U)-\dim(V),\]
and it follows from \eqref{sfl} and \eqref{sfl-equiv} that

\begin{align}\label{forgetfull}
F(\sfl_G(\mathcal{A}))=\sfl(\mathcal{A}).
\end{align}
Consequently, the classical spectral flow of $\mathcal{A}$ has to vanish if $\sfl_G(\mathcal{A})$ is trivial. We will see below in Section \ref{sect-simpEx} a simple example of a path of $\mathbb{Z}_2$-equivariant operators such that $\sfl_G(\mathcal{A})$ is non-trivial even though $\sfl(\mathcal{A})=0$.


\subsection{The $G$-equivariant Spectral Flow: Properties}

The aim of this section is to show that the $G$-equivariant spectral flow \eqref{sfl-equiv} has the same formal properties than the classical spectral flow \eqref{sfl}. Henceforth we denote by $\mathcal{FS}_G(W,H)$ the subset of $\mathcal{FS}(W,H)$ of all $G$-equivariant operators.

\begin{lemma}\label{sfl-concatenation}
If $\mathcal{A}^1$ and $\mathcal{A}^2$ are two paths in $\mathcal{FS}_G(W,H)$ such that $\mathcal{A}^1_1=\mathcal{A}^2_0$, then

\[\sfl_G(\mathcal{A}^1\ast\mathcal{A}^2)=\sfl_G(\mathcal{A}^1)+\sfl_G(\mathcal{A}^2)\in RO(G),\]
where $\mathcal{A}^1\ast\mathcal{A}^2$ denotes the concatenation of paths in $\mathcal{FS}_G(W,H)$.
\end{lemma}

\begin{proof}
This is an immediate consequence of the definition \eqref{sfl-equiv}.
\end{proof}
\noindent
The following lemma is the existence property of the spectral flow, which means that a non-trivial spectral flow yields a non-trivial kernel. Here we denote the set of all bounded invertible operators by $GL(W,H)$. 

\begin{lemma}\label{sfl-zero}
If $\mathcal{A}_\lambda\in GL(W,H)\cap\mathcal{FS}_G(W,H)$ for all $\lambda\in I$, then $\sfl_G(\mathcal{A})=0\in RO(G)$.
\end{lemma}

\begin{proof}
As $GL(W,H)\subset\mathcal{L}(W,H)$ is open and the interval $I$ is compact, there is $\varepsilon>0$ such that $\sigma(\mathcal{A}_\lambda)\cap[-\varepsilon,\varepsilon]=\emptyset$ for all $\lambda\in I$. Now the assertion is a consequence of Lemma \ref{lemma-technical} and \eqref{sfl-equiv}.  
\end{proof}
\noindent
The next lemma is important for reducing spectral flow computations to finite dimensions (see, e.g., \cite{SFLPejsachowiczII}, \cite{Edinburgh}).

\begin{lemma}\label{lemma-splitting}
Let $H=H_1\oplus H_2$, where $H_1$, $H_2$ are $G$-invariant and such that $\mathcal{A}_\lambda\mid_{H_i}\in\mathcal{FS}(W_i,H_i)$ for $i=1,2$,  $\lambda\in I$, and suitable $W_i\subset H_i$. Then

\[\sfl_G(\mathcal{A})=\sfl_G(\mathcal{A}\mid_{H_1})+\sfl_G(\mathcal{A}\mid_{H_2})\in RO(G).\] 
\end{lemma}

\begin{proof}
It follows from the assumptions that $E(\mathcal{A}_\lambda,[a,b])=E(\mathcal{A}_\lambda\mid_{H_1},[a,b])\oplus E(\mathcal{A}_\lambda\mid_{H_2},[a,b])$, where both spaces on the right hand side are $G$-invariant. Using this in \eqref{sfl-equiv} yields the claimed equality for the $G$-equivariant spectral flow.
\end{proof}
\noindent
The following lemma clarifies the dependence of the spectral flow on the orientation of the path.

\begin{lemma}\label{sfl-reverse}
If $\mathcal{A}'_\lambda:=\mathcal{A}_{1-\lambda}$, $\lambda\in I$, for a path $\mathcal{A}$ in $\mathcal{FS}_G(W,H)$, then

\[\sfl_G(\mathcal{A}')=-\sfl_G(\mathcal{A})\in RO(G).\]
\end{lemma}

\begin{proof}
This is an immediate consequence of \eqref{sfl-equiv} and the fact that\linebreak $[E(\mathcal{A}_{\lambda_{i-1}},[0,a_i])]-[E(\mathcal{A}_{\lambda_{i}},[0,a_i])]$ is the inverse of $[E(\mathcal{A}_{\lambda_i},[0,a_i])]-[E(\mathcal{A}_{\lambda_{i-1}},[0,a_i])]$ in $RO(G)$.
\end{proof}
\noindent
Our final aim of this section is to show the homotopy invariance of the $G$-equivariant spectral flow. Here we follow the approach of \cite{CompSfl}, which is based on \cite{Phillips}, and begin by the following proposition.

\begin{prop}\label{prop-hominv}
If $h:I\times I\rightarrow\mathcal{FS}_G(W,H)$ is a homotopy, then

\[\sfl_G(h(0,\cdot))=\sfl_G(h(\cdot,0))+\sfl_G(h(1,\cdot))-\sfl_G(h(\cdot,1))\in RO(G).\]
\end{prop}

\begin{proof}
As $h(I\times I)\subset\mathcal{FS}_G(W,H)$ is compact, we can find an open cover of this set by finitely many open sets $N_i$, $i=1,\ldots,n$, of the type as in Lemma \ref{lemma-sflconstruction}. This means that for every $i=1,\ldots,n$, there is $a_i>0$ such that for all $T\in N_i$, $\pm a_i\notin\sigma(T)$, $[-a_i,a_i]\cap\sigma_{ess}(T)=\emptyset$, and

\[N_i\ni T\mapsto\chi_{[-a_i,a_i]}(T)\in\mathcal{L}(H)\]
is continuous. The preimages $h^{-1}(N_i)$, $i=1,\ldots,n$ are an open cover of $I\times I$, and consequently there is some $\varepsilon>0$ such that each subset of $I\times I$ of diameter less than $\varepsilon$ is contained in one of the $h^{-1}(N_i)$.\\
Let $0=\lambda_0\leq\ldots\leq \lambda_m=1$ be a partition of $I$ such that $|\lambda_i-\lambda_{i-1}|\leq\frac{\varepsilon}{\sqrt{2}}$ for $1\leq i\leq m$. Then each $h([\lambda_{i-1},\lambda_i]\times[\lambda_{j-1},\lambda_j])$ is contained in one of the sets $N_k$. We now consider the four paths obtained from the boundary of the square $[\lambda_{i-1},\lambda_i]\times[\lambda_{j-1},\lambda_j]$, i.e. the two horizontal paths

\begin{align*}
h^h_{i-1,j}(\lambda)=h(\lambda_{i-1},\lambda),\, \lambda\in[\lambda_{j-1},\lambda_j],\qquad h^h_{i,j}(\lambda)=h(\lambda_{i},\lambda),\, \lambda\in[\lambda_{j-1},\lambda_j],
\end{align*}   
and the two vertical paths

\begin{align*}
h^v_{i,j-1}(\lambda)=h(\lambda,\lambda_{j-1}),\, \lambda\in[\lambda_{i-1},\lambda_i],\qquad  h^v_{i,j}(\lambda)=h(\lambda,\lambda_{j}),\, \lambda\in[\lambda_{i-1},\lambda_i].
\end{align*}
We denote as in Lemma \ref{sfl-reverse} by $(h^v_{i,j})'$ the reverse path of $h^v_{i,j}$. Now $h^v_{i,j-1}\ast h^h_{i,j}\ast (h^v_{i,j})'$ and $h^h_{i-1,j}$ are paths in $N_k$ having the same initial and endpoints. We obtain from \eqref{sfl-equiv} and Lemma \ref{sfl-concatenation}

\begin{align}\label{equ-proofhomotopy}
\begin{split}
\sfl_G(h^h_{i-1,j})&=[E((h^h_{i-1,j})_1,[0,a_k])]-[E((h^h_{i-1,j})_0,[0,a_k])]\\
&=[E((h^v_{i,j-1}\ast h^h_{i,j}\ast (h^v_{i,j})')_1,[0,a_k])]-[E((h^v_{i,j-1}\ast h^h_{i,j}\ast (h^v_{i,j})')_0,[0,a_k])]\\
&=\sfl_G(h^v_{i,j-1}\ast h^h_{i,j}\ast (h^v_{i,j})')=\sfl_G(h^v_{i,j-1})+\sfl_G(h^h_{i,j})+\sfl_G((h^v_{i,j})')\\
&=\sfl_G(h^v_{i,j-1})+\sfl_G(h^h_{i,j})-\sfl_G(h^v_{i,j}),
\end{split}
\end{align} 
where we have used Lemma \ref{sfl-reverse} in the final equality. Consequently, 

\begin{align*}
\sfl_G(h(0,\cdot))&=\sum^m_{j=1}{\sfl_G(h^h_{0,j})}=\sum^m_{j=1}{\left(\sfl_G(h^v_{1,j-1})+\sfl_G(h^h_{1,j})-\sfl_G(h^v_{1,j})\right)}\\
&=\sfl_G(h^v_{1,0})-\sfl_G(h^v_{1,m})+\sum^m_{j=1}{\sfl_G(h^h_{1,j})}.
\end{align*}
Now, again by \eqref{equ-proofhomotopy},

\begin{align*}
\sum^m_{j=1}{\sfl_G(h^h_{1,j})}&=\sum^m_{j=1}{\left(\sfl_G(h^v_{2,j-1})+\sfl_G(h^h_{2,j})-\sfl_G(h^v_{2,j})\right)}=\sfl_G(h^v_{2,0})-\sfl_G(h^v_{2,m})+\sum^m_{j=1}{\sfl_G(h^h_{2,j})},
\end{align*}
and we obtain

\begin{align*}
\sfl_G(h(0,\cdot))&=\sfl_G(h^v_{1,0})+\sfl_G(h^v_{2,0})-\sfl_G(h^v_{1,m})-\sfl_G(h^v_{2,m})+\sum^m_{j=1}{\sfl_G(h^h_{2,j})}.
\end{align*}
If we continue until $i=m$, we get

\begin{align*}
\sfl_G(h(0,\cdot))&=\sum^m_{i=1}{\sfl_G(h^v_{i,0})}-\sum^m_{i=1}{\sfl_G(h^v_{i,m})}+\sum^m_{j=1}{\sfl_G(h^h_{m,j})}\\
&=\sfl_G(h(\cdot,0))-\sfl_G(h(\cdot,1))+\sfl_G(h(1,\cdot)),
\end{align*}
where we have used Lemma \ref{sfl-concatenation} in the final equality. This is the claimed equation.
\end{proof}
\noindent
As the $G$-equivariant spectral flow is trivial for constant paths, we immediately obtain the homotopy invariance under homotopies having constant endpoints.

\begin{cor}
If $h:I\times I\rightarrow\mathcal{FS}_G(W,H)$ is a homotopy such that $h(\cdot,0)$ and $h(\cdot,1)$ are constant, then

\[\sfl_G(h(0,\cdot))=\sfl_G(h(1,\cdot))\in RO(G).\]
\end{cor}
\noindent
Moreover, Lemma \ref{sfl-zero} yields the invariance under homotopies having invertible endpoints.

\begin{cor}
If $h:I\times I\rightarrow\mathcal{FS}_G(W,H)$ is a homotopy such that $h(\lambda,0)$ and $h(\lambda,1)$ are invertible for all $\lambda\in I$, then

\[\sfl_G(h(0,\cdot))=\sfl_G(h(1,\cdot))\in RO(G).\]
\end{cor}
\noindent
Finally, let us point out that further homotopy invariance properties can be obtained from Proposition \ref{prop-hominv}, e.g., the invariance under free homotopies of loops in $\mathcal{FS}_G(W,H)$ (cf. \cite{CompSfl}).


\subsection{A Simple Example}\label{sect-simpEx}
In this section we consider a path $\widetilde{\mathcal{A}}=\{\widetilde{\mathcal{A}_\lambda}\}_{\lambda\in I}$ in $\mathcal{FS}(W,H)$, where $W$ and $H$ are as before Hilbert spaces with a dense embedding $W\hookrightarrow H$, but we do not assume that there is a group action on $H$. The path $\mathcal{A}=\{\mathcal{A}_\lambda\}_{\lambda\in I}$ for $\mathcal{A}_\lambda:W\times W\rightarrow H\times H$, where

\[\mathcal{A}_\lambda=\begin{pmatrix}
\widetilde{\mathcal{A}}_\lambda&0\\
0&-\widetilde{\mathcal{A}}_\lambda
\end{pmatrix},\]
is presumably the easiest type of a path having a vanishing spectral flow for symmetry reasons. Note that indeed $\sfl(\mathcal{A})=0$ by \eqref{sfl} as the eigenvalues and their multiplicities of any $\mathcal{A}_\lambda$ are symmetric about $0$.\\
As all irreducible real representations of $\mathbb{Z}_2$ are one dimensional, every real $k$-dimensional representation is up to isomorphism a $k\times k$ diagonal matrix of the form $\diag(1,\ldots,1,-1,\ldots,-1)$. Thus we obtain an isomorphism $\phi:RO(\mathbb{Z}_2)\rightarrow\mathbb{Z}\oplus\mathbb{Z}$ of abelian groups by setting
 
 \begin{align}\label{phi}
 \phi([E]-[F])=(\dim(E)-\dim(F),\dim(E_G)-\dim(F_G)),
 \end{align}
 where $E_G\subset E$ and $F_G\subset F$ denote the spaces of fixed points under the group action.\\
Now let us consider on $H\times H$ the action of $G=\mathbb{Z}_2=\{e,g\}$ given by $g(u,v)=(u,-v)$. The aim of this section is to show that under the identification $RO(G)=\mathbb{Z}\oplus\mathbb{Z}$ by \eqref{phi}, the equivariant spectral flow of $\mathcal{A}$ is

\begin{align}\label{simpleexample}
\sfl_G(\mathcal{A})=(\sfl(\mathcal{A}),\sfl(\widetilde{\mathcal{A}}))=(0,\sfl(\widetilde{\mathcal{A}}))
\end{align}
and so it is non-trivial if and only if $\sfl(\widetilde{\mathcal{A}})$ does not vanish.\\
Let $0=\lambda_0< \lambda_1<\ldots< \lambda_N=1$ and $a_i>0$ be as in the definition \eqref{sfl} for the spectral flow of the path $\widetilde{\mathcal{A}}$. Then, as $\sigma_p(\mathcal{A}_\lambda)=\sigma_p(\widetilde{\mathcal{A}}_\lambda)\cup\sigma_p(-\widetilde{\mathcal{A}}_\lambda)$ and $\sigma_{ess}(\mathcal{A}_\lambda)=\sigma_{ess}(\widetilde{\mathcal{A}}_\lambda)\cup\sigma_{ess}(-\widetilde{\mathcal{A}}_\lambda)$, we see that the same numbers $\lambda_0,\ldots, \lambda_N$ and $a_1,\ldots,a_N$ can be used for $\mathcal{A}$ in \eqref{sfl-equiv}. We have for $1\leq i\leq N$,

\begin{align*}
[E(\mathcal{A}_{\lambda_i},[0,a_i])]-[E(\mathcal{A}_{\lambda_{i-1}},[0,a_i])]&=[E(\widetilde{\mathcal{A}}_{\lambda_i},[0,a_i])\times E(-\widetilde{\mathcal{A}}_{\lambda_i},[0,a_i])]\\
&-[E(\widetilde{\mathcal{A}}_{\lambda_{i-1}},[0,a_i])\times E(-\widetilde{\mathcal{A}}_{\lambda_{i-1}},[0,a_i])])\\
&=[E(\widetilde{\mathcal{A}}_{\lambda_i},[0,a_i])\times E(\widetilde{\mathcal{A}}_{\lambda_i},[-a_i,0])]\\
&-[E(\widetilde{\mathcal{A}}_{\lambda_{i-1}},[0,a_i])\times E(\widetilde{\mathcal{A}}_{\lambda_{i-1}},[-a_i,0])]).
\end{align*}
Hence we obtain from \eqref{phi} that $\phi([E(\mathcal{A}_{\lambda_i},[0,a_i])]-[E(\mathcal{A}_{\lambda_{i-1}},[0,a_i])])$ is given by

\begin{align*}
(\dim(E(\mathcal{A}_{\lambda_i},[0,a_i]))-\dim(E(\mathcal{A}_{\lambda_{i-1}},[0,a_i])), \dim(E(\widetilde{\mathcal{A}}_{\lambda_i},[0,a_i]))-\dim(E(\widetilde{\mathcal{A}}_{\lambda_{i-1}},[0,a_i]))).
\end{align*}
Now \eqref{simpleexample} follows from the definitions \eqref{sfl} and \eqref{sfl-equiv}.

\section{Applications in Bifurcation Theory of Critical Points of Functionals with Symmetries}


\subsection{$G$-equivariant Spectral Flow, Morse Index and Bifurcation}
We consider equations of the type $\nabla f_\lambda(u)=0$, where $f:I\times H\rightarrow \mathbb{R}$ is a family of $C^2$-functionals on an infinite dimensional real Hilbert space $H$, and we assume that $\nabla f_\lambda(0)=0$ for all $\lambda\in I$, i.e. $0\in H$ is a critical point of all functionals $f_\lambda$. A \textit{bifurcation point} is a parameter value $\lambda^\ast\in I$ at which non-trivial critical points branch off from the trivial ones $I\times\{0\}$, i.e., in every neighbourhood of $(\lambda^\ast,0)\in I\times H$ there is some $(\lambda,u)$ such that $\nabla f_\lambda(u)=0$ and $u\neq 0$. A crucial role for studying the existence of bifurcation points is played by the family of Hessians $L_\lambda:=D^2_0f_\lambda$ at $0\in H$, which are bounded selfadjoint operators on $H$. Note that, by the implicit function theorem, $L_\lambda$ is non-invertible if $\lambda$ is a bifurcation point. However, it is not difficult to see that the non-invertibility of $L_\lambda$ is not sufficient for the existence of bifurcation points.\\
It is a common assumption that $L_\lambda=D^2_0f_\lambda$ are Fredholm operators, i.e., $L_\lambda\in\mathcal{FS}(H):=\mathcal{FS}(H,H)$. It was shown by Atiyah and Singer in \cite{AtiyahSinger} that $\mathcal{FS}(H)$ has the three connected components

\begin{align*}
\mathcal{FS}_+(H)&:=\{T\in\mathcal{FS}(H):\, \sigma_{ess}(T)\subset(0,+\infty)\}\\
\mathcal{FS}_-(H)&:=\{T\in\mathcal{FS}(H):\, \sigma_{ess}(T)\subset(-\infty,0)\},
\end{align*}

and
\begin{align*}
\mathcal{FS}_\ast(H):=\mathcal{FS}(H)\setminus\mathcal{FS}_\pm(H).
\end{align*}
The operators in $\mathcal{FS}_+(H)$ have a finite Morse index

\begin{align}\label{Morse}
\mu_-(L_\lambda)=\dim\left(\oplus_{\mu<0}\{u\in H:\, L_\lambda u=\mu u\}\right)
\end{align}
i.e. they have only finitely many negative eigenvalues including multiplicities. A well-known theorem in nonlinear analysis says that, if $L_\lambda\in\mathcal{FS}_+(H)$, $\lambda\in I$,  $L_0$ and $L_1$ are invertible and have different Morse indices,
 
 \begin{align}\label{Smoller-Wasserman}
 \mu_-(L_0)\neq\mu_-(L_1),
 \end{align}  
then there is a bifurcation of critical points from the trivial branch (see \cite{Mawhin}, \cite{SmollerWasserman}).\\ 
Smoller and Wasserman considered in \cite{SmollerWasserman} a finer invariant than the Morse index in \eqref{Smoller-Wasserman} to study bifurcation for equations with symmetries. Let us consider as before a family of functionals $f:I\times H\rightarrow\mathbb{R}$, but we now require in addition that each $f_\lambda$ is $G$-invariant, i.e., $f(g u)=f(u)$, $g\in G$, where $G$ is a compact Lie group acting orthogonally on $H$. Under this assumption, the selfadjoint operators $L_\lambda=D^2_0f_\lambda$ are $G$-equivariant, i.e. $L_\lambda(gu)=g(L_\lambda u)$. If we now again assume that $L_\lambda\in\mathcal{FS}_+(H)$, then the spaces 

\begin{align}\label{equMorse}
E^-(L_\lambda)=\oplus_{\mu<0}\{u\in H:\, L_\lambda u=\mu u\}
\end{align}
in \eqref{Morse} are of finite dimension and so we obtain finite dimensional real representations of the compact Lie group $G$. Henceforth, we assume that $G$ is \textit{nice} in the sense of \cite{SmollerWasserman}, i.e., any orthogonal representations $E$ and $F$ of $G$ are isomorphic if the quotients $D(E)/S(E)$ and $D(F)/S(F)$ of the unit discs by the unit spheres have the same $G$-equivariant homotopy type. For example, if $G_0$ denotes the connected component of the identity in $G$, then $G$ is nice if $G/G_0$ is trivial or a finite product of $\mathbb{Z}_2$ or $\mathbb{Z}_3$. Thus, in particular, $S^1$, $O(n)$ and $SO(n)$ are nice. The following theorem is the main result of \cite{SmollerWasserman}.

\begin{theorem}\label{thm:Smoller-Wasserman}
If $L_\lambda\in\mathcal{FS}_+(H)\cap\mathcal{FS}_G(H)$, $\lambda\in I$, $L_0$,$L_1$ are invertible and $E^-(L_0)$ and $E^-(L_1)$ are not isomorphic as $G$-representations, then there is a bifurcation of critical points for $f$. 
\end{theorem}
\noindent
The aim of this section is to show that the assumption in Theorem \ref{thm:Smoller-Wasserman} can be replaced by requiring that the $G$-equivariant spectral flow of the path $L=\{L_\lambda\}_{\lambda\in I}$ in $\mathcal{FS}_G(H)$ is a non-trivial element in $RO(G)$. This follows from the next proposition.

\begin{prop}\label{prop-sfldiffMorse}
If $L=\{L_\lambda\}_{\lambda\in I}$ is a path in $\mathcal{FS}_G(H)\cap\mathcal{FS}_+(H)$, then

\[\sfl_G(L)=[E^-(L_0)]-[E^-(L_1)]\in RO(G).\]
\end{prop}

\begin{proof}
The proof falls into three steps. Let us firstly assume that there is $[c,d]\subset I$ and some $a>0$ such that $\pm a\notin\sigma(L_\lambda)$ and such that there is no essential spectrum in $[-a,a]\cap\sigma(L_\lambda)$ for all $\lambda\in [c,d]$. By Lemma \ref{lemma-technical}, $E(L_c,[-a,a])$ and $E(L_d,[-a,a])$ are isomorphic $G$-representations. Consequently,

\begin{align*}
0&=[E(L_c,[-a,a])]-[E(L_d,[-a,a])]\\
&=[E(L_c,[0,a])\oplus E(L_c,[-a,0))]-[E(L_d,[0,a])\oplus E(L_d,[-a,0))]\\
&=([E(L_c,[0,a])]-[E(L_d,[0,a]))+([E(L_c,[-a,0))]-E(L_d,[-a,0))]
\end{align*}
which means that

\begin{align}\label{Einv}
[E(L_d,[0,a])]-[E(L_c,[0,a])]=[E(L_c,[-a,0))]-[E(L_d,[-a,0))]\in RO(G).
\end{align}
Secondly, as $\mu_-(L_\lambda)<\infty$ for all $\lambda\in[c,d]$, there exists $m\geq a>0$ such that 

\begin{align}\label{specdisjoint}
(-\infty,-m]\cap\sigma(L_\lambda)=\emptyset.
\end{align}
As $- a\not\in\sigma(L_\lambda)$ and $(-\infty,0)\cap\sigma_{ess}(L_\lambda)=\emptyset$ for all $\lambda\in[c,d]$, $E(L_c,[-m,-a])$ and $E(L_d,[-m,-a])$ are isomorphic $G$-representations by Lemma \ref{lemma-technical}. Hence

\[[E(L_c,[-m,-a])]-[E(L_d,[-m,-a])]=0\in RO(G),\] 
and it follows from \eqref{Einv} that

\begin{align*}
[E(L_d,[0,a])]-[E(L_c,[0,a])]&=[E(L_c,[-a,0))]-[E(L_d,[-a,0))]\\
&+[E(L_c,[-m,-a])]-[E(L_d,[-m,-a])]\\
&=[E(L_c,[-a,0))\oplus E(L_c,[-m,-a])]\\
&-[E(L_d,[-a,0))\oplus E(L_d,[-m,-a])]\\
&=[E(L_c,[-m,0))]-[E(L_d,[-m,0))]\\
&=[E^-(L_c)]-[E^-(L_d)]\in RO(G),
\end{align*} 
where we have used \eqref{specdisjoint} in the final step.\\
Finally, we consider a partition $0=\lambda_0<\ldots<\lambda_N=1$ of $I$ and numbers $a_i>0$, $i=1,\ldots, N$, as in \eqref{sfl-equiv}. We obtain

\begin{align*}
\sfl_G(L)&=\sum^N_{i=1}{([E(L_{\lambda_i},[0,a_i])]-[E(L_{\lambda_{i-1}},[0,a_i])])}=\sum^N_{i=1}{([E^-(L_{\lambda_{i-1}})]-[E^-(L_{\lambda_i})])}\\
&=[E^-(L_0)]-[E^-(L_1)]\in RO(G),
\end{align*}  
where we have used that the last sum is telescoping.
\end{proof}
\noindent
The previous proposition and Theorem \ref{thm:Smoller-Wasserman} yield the following corollary, where again $L_\lambda$ are the Hessians of the family of $G$-invariant $C^2$-functionals $f:I\times H\rightarrow\mathbb{R}$ and $G$ is assumed to be nice.

\begin{cor}\label{Cor-SmollerWasserman}
If $L_\lambda\in\mathcal{FS}_+(H)\cap\mathcal{FS}_G(H)$, $\lambda\in I$, $L_0$,$L_1$ are invertible and $\sfl_G(L)\neq 0$, then there is a bifurcation point of critical points of $f$. 
\end{cor}

\begin{proof}
If $\sfl_G(L)\neq 0$, then 

\[[E^-(L_0)]-[E^-(L_1)]\neq 0\in RO(G)\]
by Proposition \ref{prop-sfldiffMorse}. Consequently, $E^-(L_0)$ and $E^-(L_1)$ cannot be isomorphic as $G$-representations and the assertion follows from Theorem \ref{thm:Smoller-Wasserman}.  
\end{proof}





\subsection{Bifurcation of Periodic Solutions of Autonomous Hamiltonian Systems}
There are important types of differential equations whose solutions are critical points of a functional $f_\lambda:H\rightarrow\mathbb{R}$ on a Hilbert space $H$ but where the Hessians $L_\lambda$ are in the path component $\mathcal{FS}_\ast(H)$. Thus the Morse index \eqref{Morse} cannot be used for studying bifurcation phenomena for such equations. The most prominent class of this type are Hamiltonian systems and numerous works have appeared over the last decades where approaches were developed to deal with their bifurcation problems in various settings (cf., e.g., \cite{Mawhin}, \cite{Szulkin}, \cite{BartschSzulkin}, \cite{RybickiS1}, \cite{SFLPejsachowiczII}). However, to the best of our knowledge, results in an equivariant setting as in the previous section have not been obtained. It is the main purpose of this paper to show that the $G$-equivariant spectral flow is the right invariant to deal with such bifurcation problems.

\subsubsection{Index Formula and Bifurcation Theorem}
We consider families of autonomous Hamiltonian systems

\begin{equation}\label{HamiltonianNonlin}
\left\{
\begin{aligned}
J u'(t)&+\nabla_u\mathcal{H}(\lambda,u(t))=0,\quad t\in [0,2\pi]\\
u(0)&=u(2\pi),
\end{aligned}
\right.
\end{equation}
where $J$ is the symplectic standard matrix

\begin{align}\label{J}
J=\begin{pmatrix}
0&-I_n\\
I_n&0
\end{pmatrix},
\end{align}
and $\mathcal{H}:I\times\mathbb{R}^{2n}\rightarrow\mathbb{R}$ is of the form 

\[\mathcal{H}(\lambda,u)=\frac{1}{2}\langle S_\lambda u,u\rangle+R(\lambda,u)\]
for a family $S:I\times\mathbb{R}^{2n}\rightarrow\mathbb{R}^{2n}$ of symmetric matrices and a $C^2$-map $R:I\times\mathbb{R}^{2n}\rightarrow\mathbb{R}$ such that

\begin{itemize}
 \item $\nabla_uR(\lambda,u)=o(\|u\|)$ as $u\rightarrow 0$,
 \item there are $r>0$ and $a,b>0$ such that for all $(\lambda,u)\in I\times\mathbb{R}^{2n}$
 
 \[ \|D^2_u R(\lambda,u)\|\leq a+b\|u\|^r.\]
\end{itemize}
Note that the constant function $0$ is a solution of \eqref{HamiltonianNonlin} for all $\lambda\in I$. It is well known that solutions of \eqref{HamiltonianNonlin} are the critical points of the functional $f_\lambda:H^\frac{1}{2}(S^1,\mathbb{R}^{2n})\rightarrow\mathbb{R}$ given by

\begin{align}\label{f}
f_\lambda(u)=\frac{1}{2}\Gamma(u,u)+\int^{2\pi}_0{\mathcal{H}(\lambda,u(t))\,dt},
\end{align}
where $\Gamma$ is a bounded bilinear form on $H^\frac{1}{2}(S^1,\mathbb{R}^{2n})$ such that 

\begin{align}\label{Gamma}
\Gamma(u,v)=\int^{2\pi}_0{\langle Ju'(t),v(t)\rangle\, dt}
\end{align}
if $u\in H^1(S^1,\mathbb{R}^{2n})$ (cf. \cite[\S 1]{SFLPejsachowiczII}). The Hessians $L_\lambda:H^\frac{1}{2}(S^1,\mathbb{R}^{2n})\rightarrow H^\frac{1}{2}(S^1,\mathbb{R}^{2n})$ of $f_\lambda$ at $0$ are Fredholm operators determined by

\begin{align}\label{LHam}
\langle L_\lambda u,v\rangle_{H^\frac{1}{2}}=\Gamma(u,v)+\int^{2\pi}_0{\langle S_\lambda u,v\rangle\, dt},
\end{align}  
and it is readily seen that the kernel of $L_\lambda$ is the space of solutions of 

\begin{equation}\label{Hamiltonianlin}
\left\{
\begin{aligned}
J u'(t)&+S_{\lambda}u(t)=0,\quad t\in [0,2\pi]\\
u(0)&=u(2\pi).
\end{aligned}
\right.
\end{equation}
We now consider two kinds of group actions on $H^\frac{1}{2}(S^1,\mathbb{R}^{2n})$. The first action is obtained from the periodicity. We let the compact Lie group $S^1$ act on $H^\frac{1}{2}(S^1,\mathbb{R}^{2n})$ by setting

\[\Theta u=u(\cdot+\Theta),\]  
and note that each $f_\lambda$ is $S^1$-invariant. It is well-known that the only non-trivial irreducible real representations of $S^1$ are 

\[\begin{pmatrix}
\cos(k\Theta)&-\sin(k\Theta)\\
\sin(k\Theta)&\cos(k\Theta)
\end{pmatrix}\in SO(2),\qquad k\in\mathbb{N},\]
and it is readily seen from this fact that $RO(S^1)$ and $\mathbb{Z}[x]$ are isomorphic as abelian groups by an isomorphism mapping \eqref{sfl-equiv} to a polynomial having the classical spectral flow \eqref{sfl} as zero-order term.\\
Further, we consider a type of group action that was introduced by Bartsch and Willem in \cite{BartschWillem}. Let $G_0$ be a compact Lie group acting orthogonally on $\mathbb{R}^{2n}$ and such that that $g^TJg=J$ and $gS_\lambda=S_\lambda g$ for all $g\in G_0$ and $\lambda\in I$. We obtain an induced orthogonal action on $H^\frac{1}{2}(S^1,\mathbb{R}^{2n})$ by $(gu)(t)=gu(t)$ and each $f_\lambda$ is invariant under this action if $R(\lambda,gu)=R(\lambda,u)$, $\lambda\in I$. In total, we have an orthogonal action of $G:=G_0\times S^1$ on $H^\frac{1}{2}(S^1,\mathbb{R}^{2n})$ such that each $f_\lambda$ is invariant. Consequently, $L=\{L_\lambda\}_{\lambda\in I}$ is a path of $G$-equivariant selfadjoint Fredholm operators and thus $\sfl_G(L)\in RO(G)$ is defined.\\
To state our main theorems, we define symmetric $4n\times 4n$-matrices by 

\begin{align}
A_k(\lambda):=\begin{pmatrix}
\frac{1}{k} S_\lambda&J\\
-J&\frac{1}{k} S_\lambda
\end{pmatrix},\qquad k\in\mathbb{N},
\end{align}
and denote by $E^-(A_k(\lambda))$, $\lambda\in I$, the direct sum of the eigenspaces of the matrix $A_k(\lambda)$ with respect to negative eigenvalues. Moreover, we consider for $k\in\mathbb{N}$ the representations

\begin{align}\label{R4nactionS}
\Theta\mapsto \begin{pmatrix}
\cos(k\Theta)\,I_{2n}&-\sin(k\Theta)\,I_{2n}\\
\sin(k\Theta)\,I_{2n}&\cos(k\Theta)\, I_{2n}
\end{pmatrix}\in SO(4n)
\end{align}
of $S^1$, and 

\begin{align}\label{R4nactionG}
g\mapsto \begin{pmatrix}
g&0\\
0&g
\end{pmatrix}\in O(4n)
\end{align}
of $G_0$ on $\mathbb{R}^{4n}$. Note that $A_k(\lambda)$ is equivariant under these group actions, and thus the direct sums of negative eigenspaces $E^-(A_k(\lambda))\subset\mathbb{R}^{4n}$ are invariant spaces.

\begin{theorem}\label{thm-Hamiltonian}
For all but finitely many $k\in\mathbb{N}$,

\[[E^-(A_k(0))]-[E^-(A_k(1))]=0\in RO(G),\]
and

\begin{align*}
\sfl_{G}(L)=[E^-(S_0)]-[E^-(S_1)]+\sum^{\infty}_{k=1}{([E^-(A_k(0))]-[E^-(A_k(1))])}\in RO(G).
\end{align*}
\end{theorem}
\noindent
Here $E^-(S_0)$ and $E^-(S_1)$ are the negative eigenspaces of the $2n\times 2n$ matrices $S_{0/1}$, on which $S^1$ acts trivially and $G_0$ by its action on $\mathbb{R}^{2n}$. 

\begin{theorem}\label{thm-HamiltonianNonlin}
If $G$ is nice, \eqref{Hamiltonianlin} has only the trivial solution for $\lambda=0,1,$ and $$\sfl_G(L)\neq 0\in RO(G),$$ then there is a bifurcation from the trivial branch for \eqref{HamiltonianNonlin}.
\end{theorem}
\noindent
When we consider a trivial $G_0$ action and apply \eqref{forgetfull} to $\sfl_G(L)$, we obtain from Theorem \ref{thm-Hamiltonian} and Theorem \ref{thm-HamiltonianNonlin} the following corollary, which is Theorem 1.1 of \cite{SFLPejsachowiczII}.

\begin{cor}
If \eqref{Hamiltonianlin} has only the trivial solution for $\lambda=0,1,$ and 

\[\dim(E^-(S_0))-\dim(E^-(S_1))+\sum^{\infty}_{k=1}{(\dim(E^-(A_k(0)))-\dim(E^-(A_k(1))))}\neq 0\in\mathbb{Z},\]
then there is a bifurcation from the trivial branch for \eqref{HamiltonianNonlin}.
\end{cor}
\noindent
Let us note that the previous corollary was also recently obtained by B\l{}aszczyk, Go\l{}\k{e}biewska and Rybicki in \cite[Thm.5.1]{BlaGoRy} by using the infinite-dimensional Conley index from \cite{Marek}.

\subsubsection{Proof of Theorem \ref{thm-Hamiltonian}}
At first, let us recall that

\[H:=H^\frac{1}{2}(S^1,\mathbb{R}^{2n})=\left\{\sum^\infty_{k=0}{\sin(k\cdot)a_k+\cos(k\cdot)b_k}:\, a_k, b_k\in\mathbb{R}^{2n},\sum^\infty_{k=1}{k(|a_k|^2+|b_k|^2)}<\infty\right\}\]
with the usual scalar product.\\
We consider for $k\geq 0$ the spaces

\[V_k:=\{(\sin kt)a+(\cos kt)b:\, a,b\in\mathbb{R}^{2n}\}\subset H,\]
and note that each $V_k$ is invariant under the action of $G=G_0\times S^1$. Moreover, it follows from \eqref{Gamma} and \eqref{LHam} that each space $V_k$ reduces the operators $L_\lambda$ (cf. \cite[\S 1]{SFLPejsachowiczII}). Let $e_1,\ldots,e_n,e_{n+1},\ldots,e_{2n}$ be the standard basis of $\mathbb{R}^{2n}$. This yields a basis of $V_k$ for $k\in\mathbb{N}$ by 

\begin{align}\label{basis}
\{u^k_1,\ldots,u^k_{2n},v^k_1,\ldots,v^k_{2n}\},
\end{align}
where $u^k_i=\sin(kt)e_i$ and $v^k_i=\cos(kt)e_i$ for $i=1,\ldots,2n$.
As $Je_i=e_{i+n}$ for $i=1,\ldots,n$, it is readily seen from \eqref{Gamma} and \eqref{LHam} that $L_\lambda\mid_{V_K}$, $k\in\mathbb{N}$, is given with respect to the basis \eqref{basis} by the $4n\times 4n$-matrix

\begin{align}\label{Ak}
A_k(\lambda):=\begin{pmatrix}
\frac{1}{k} S_\lambda& J\\
-J&\frac{1}{k} S_\lambda
\end{pmatrix}.
\end{align}
Moreover, $\{e_1,\ldots,e_{2n}\}$ is a basis of $V_0$ and $L_\lambda\mid_{V_0}$ is given by multiplication by $S_\lambda$. The action of $S^1$ on $V_0$ is trivial. A straightforward computation shows that the actions of $S^1$ and $G_0$ on $V_k$, $k\in\mathbb{N}$, are given with respect to the basis \eqref{basis} by \eqref{R4nactionS} and \eqref{R4nactionG}.\\
Our aim is to find a decomposition $H=X\oplus Y$ into closed $G$-invariant subspaces that reduce the operators $L_\lambda$ and are such that $\dim(X)<\infty$, as well as $L_\lambda\mid_{Y}\in GL(Y)$, $\lambda\in I$.\\
Let $m_0\in\mathbb{N}$ be such that $A_k(\lambda)$ is invertible for all $k> m_0$ and $\lambda\in I$. Then the operators $L_\lambda\mid_{V_k}:V_k\rightarrow V_k$ are invertible as well for $k> m_0$. We now consider the spaces $X=\bigoplus^{m_0}_{k=0}V_k$ and $Y=X^\perp$. The operators $L_\lambda$ are reduced by the decomposition $H=X\oplus Y$ and we obtain from Lemma \ref{lemma-splitting}

\[\sfl_G(L)=\sfl_G(L\mid_{X})+\sfl_G(L\mid_{Y}).\]
As $L_\lambda\mid_{Y}\in\mathcal{FS}(Y)$ and $L_\lambda\mid_{V_k}:V_k\rightarrow V_k$ is invertible for $k\geq m_0+1$, it follows that $L_\lambda\mid_{Y}$ is invertible for $\lambda\in I$. Consequently, by Lemma \ref{sfl-zero},

\begin{align}\label{Hamiltonian-finitedimred}
\sfl_{G}(L)=\sfl_{G}(L\mid_{X}),
\end{align} 
and so we have reduced the spectral flow computation to finite dimensions. Moreover, we obtain from Lemma \ref{lemma-splitting} and Lemma \ref{sfl-zero}

\begin{align}\label{hamiltonianproofsfl}
\sfl_{G}(L)=\sum^{m_0}_{k=0}{\sfl_{G}(L\mid_{V_k})}=\sum^{\infty}_{k=0}{\sfl_{G}(L\mid_{V_k})},
\end{align}
where we have used once again that $L_\lambda\mid_{V_k}:V_k\rightarrow V_k$ is invertible for $k> m_0$.\\
Let us now consider $L\mid_{V_k}$ for some $k=0,1,2,\ldots$. As $V_k$ is of finite dimension, there is a single $a>0$ in \eqref{sfl-equiv} such that $\sigma(L_\lambda\mid_{V_k})\subset[-a,a]$ and all elements in $\sigma(L_\lambda\mid_{V_k})$ are eigenvalues of finite multiplicity. By Lemma \ref{lemma-technical}, we obtain

\begin{align*}
0&=[E(L_{1}\mid_{V_k},[-a,a])]-[E(L_{0}\mid_{V_k},[-a,a])]\\
&=[E(L_{1}\mid_{V_k},[0,a])\oplus [E(L_{1}\mid_{V_k},[-a,0))]-[E(L_{0}\mid_{V_k},[0,a])\oplus E(L_{0}\mid_{V_k},[-a,0))]\\
&=([E(L_{1}\mid_{V_k},[0,a])]-[E(L_{0}\mid_{V_k},[0,a])])+([E(L_{1}\mid_{V_k},[-a,0))]-[E(L_{0}\mid_{V_k},[-a,0))]),
\end{align*}
which shows that

\[[E(L_{1}\mid_{V_k},[0,a])]-[E(L_{0}\mid_{V_k},[0,a])]=[E(L_{0}\mid_{V_k},[-a,0))]-[E(L_{1}\mid_{V_k},[-a,0))]\in RO(G).\]
Plugging this into \eqref{sfl-equiv} yields

\begin{align*}
\begin{split}
\sfl_{G}(L\mid_{V_k})&=[E(L_{1}\mid_{V_k},[0,a])]-[E(L_{0}\mid_{V_k},[0,a])]=[E(L_{0}\mid_{V_k},[-a,0))]-[E(L_{1}\mid_{V_k},[-a,0))]\\
&=[E^-(L_{0}\mid_{V_k})]-[E^-(L_{1}\mid_{V_k})],
\end{split}
\end{align*}
where $E^-(L_{i}\mid_{V_k})$, $i=1,2$, denotes the direct sum of the eigenspaces with respect to negative eigenvalues. Consequently, with respect to the basis \eqref{basis},

\begin{align}\label{hamiltonianproofend}
\begin{split}
\sfl_{G}(L\mid_{V_k})&=[E^-(A_k(0))]-[E^-(A_k(1))]\in RO(G),\quad k\in\mathbb{N},\\
\sfl_{G}(L\mid_{V_0})&=[E^-(S_0)]-[E^-(S_1)]\in RO(G),
\end{split}
\end{align}
where $E^-(A)$ is the direct sum of the eigenspaces of a matrix $A$ with respect to negative eigenvalues and we need to consider the group actions as described below \eqref{Ak}.\\
We now see from \eqref{hamiltonianproofend} that $[E^-(A_k(0))]-[E^-(A_k(1))]=0\in RO(G)$ for $k> m_0$ as $\sfl_{G}(L\mid_{V_k})=0$ in this case. This shows the first part of Theorem \ref{thm-Hamiltonian}. Finally, the spectral flow formula is a direct consequence of \eqref{hamiltonianproofsfl} and \eqref{hamiltonianproofend}.

\subsubsection{Proof of Theorem \ref{thm-HamiltonianNonlin}}
The proof of Theorem \ref{thm-HamiltonianNonlin} is based on the following parametrised $G$-equivariant implicit function theorem.

\begin{lemma}\label{lemma-implict}
Let $H$ be a real Hilbert space and $G$ a compact Lie group acting orthogonally on $H$. Let $U\subset H$ be an open invariant subset of $H$ containing $0\in U$ and $f:I\times U\rightarrow\mathbb{R}$ a continuous one-parameter family of $G$-invariant $C^2$-functionals. Let $F(\lambda,u):=(\nabla f_\lambda)(u)$ and assume that $F(\lambda,0)=0$ for all $\lambda\in I$. Suppose that there is an orthogonal decomposition $H=X\oplus Y$, where $X$ is $G$-invariant and of finite dimension, and such that for 

\[F(\lambda,u)=(F_1(\lambda,x,y),F_2(\lambda,x,y))\in X\oplus Y,\quad u=(x,y)\in X\oplus Y,\]
we have that $(D_y F_2)(\lambda,0,0):Y\rightarrow Y$ is invertible for all $\lambda\in I$. Then:

\begin{enumerate}
 \item[(i)] There are an open ball $B_X=B(0,\delta)\subset X$ and a unique continuous family of equivariant $C^1$-maps $\eta:I\times B_X\rightarrow Y$ such that $\eta(\lambda,0)=0$ for all $\lambda\in I$, and 
 
 \begin{align}\label{implicit}
 F_2(\lambda,x,\eta(\lambda,x))=0,\quad (\lambda,x)\in I\times B_X.
 \end{align}
 
 \item[(ii)] Let the family of functionals $\overline{f}:I\times B_X\rightarrow\mathbb{R}$ and the map $\overline{F}:I\times B_X\rightarrow X$ be defined by
 
 \[\overline{f}(\lambda,x)=f(\lambda,x,\eta(\lambda,x)),\qquad \overline{F}(\lambda,x)=F_1(\lambda,x,\eta(\lambda,x)).\]
 Then $\overline{f}$ is a continuous family of $G$-invariant $C^2$-functionals on $B_X$ and $$\nabla\overline{f}(\lambda,x)=\overline{F}(\lambda,x),\qquad (\lambda,x)\in I\times B_X,$$
 which is a $G$-equivariant map.
\end{enumerate}
\end{lemma}

\begin{proof}
The lemma is proved in the non-equivariant case in \cite[Lemma 4.2]{BifJac}. Note that $F_1(\lambda,u)=PF(\lambda,u)$ and $F_2(\lambda,u)=(I_H-P)F(\lambda,u)$, where $P:H\rightarrow H$ is the orthogonal projection onto $X$. As $X$ is invariant and $G$ acts orthogonally, $P$ and $I_H-P$ are equivariant. Clearly, each $F(\lambda,\cdot)$ is equivariant since $f_\lambda$ is invariant. Thus we have for $g\in G$ and $u\in H$

\begin{align*}
F_i(\lambda,gu)=gF_i(\lambda, u),\quad i=1,2,
\end{align*}
i.e., $F_1$ and $F_2$ are families of $G$-equivariant maps. Note that, as $X$ and $Y$ are invariant, $g(x,y)=(gx,gy)$ for $g\in G$ and $(x,y)\in X\oplus Y=H$. As the lemma holds in the non-equivariant case, there is a unique family of $C^1$-maps $\eta:I\times B_X\rightarrow Y$ that satisfies \eqref{implicit}. Now $B_X$ is invariant as $G$ acts orthogonally, and thus

\begin{align*}
F_2(\lambda,gx,\eta(\lambda,gx))&=0,\quad g\in G,\, x\in B_X\\
F_2(\lambda,gx,g \eta(\lambda,x))&=gF_2(\lambda,x,\eta(\lambda,x))=0, \quad g\in G,\, x\in B_X.
\end{align*}
As $\eta$ is unique, this implies $\eta(\lambda,gx)=g \eta(\lambda,x)$ for all $g\in G$ and $x\in B_X$, i.e., each $\eta(\lambda,\cdot)$ is equivariant. Hence (i) is shown. For (ii), just note that $\overline{f}$ is invariant as $\eta$ is equivariant and $f$ is invariant. Thus $\overline{F}(\lambda,\cdot)=\nabla\overline{f}(\lambda,\cdot)$ is equivariant.  
\end{proof}
\noindent
Now recall that we obtained in the proof of Theorem \ref{thm-Hamiltonian} a decomposition $H=X\oplus Y$ into $G$-equivariant subspaces such that the operators $L_\lambda\mid_{Y}:Y\rightarrow Y$ are invertible. Let $P:H\rightarrow H$ be the orthogonal projection onto $X$. If we set 

\[F_{1}(\lambda,x,y):=PF(\lambda,x,y),\qquad F_{2}(\lambda,x,y):=(I_H-P)F(\lambda,x,y),\]
then $D_yF_{2}(\lambda,0,0)=L_\lambda\mid_{Y}$ is invertible. Thus we obtain from Lemma \ref{lemma-implict} a family $\overline{f}:I\times B_{X}\rightarrow\mathbb{R}$ of $G$-invariant $C^2$-functionals on the finite dimensional space $X$. Note that $\overline{f}$ has the same bifurcation points than $f$. As in Lemma \ref{lemma-implict}, we denote by $\overline{F}:I\times B_{X}\rightarrow X$ the family of gradients of $\overline{f}$.\\
Let $\eta_\lambda:B_{X}\rightarrow Y$ be the continuous family of equivariant $C^1$-maps from Lemma \ref{lemma-implict} for the splitting $H=X\oplus Y$. By differentiating \eqref{implicit} implicitly, we obtain

\[D_0\eta_\lambda=-(D_y F_{2}(\lambda,0,0))^{-1} D_x F_{2}(\lambda,0,0)=-(L_\lambda\mid_{Y})^{-1}(I_H-P)L_{\lambda}\mid_X=0,\]
where we use that the operators $L_\lambda$ are reduced by the decomposition $H=X\oplus Y$ and consequently $(I_H-P)L_{\lambda}\mid_X=0$. 
Thus, if $\ell_{\lambda}:=D^2_0\overline{f}_\lambda$, $\lambda\in I$, are the Hessians of $\overline{f}_\lambda$ at the critical point $0\in B_{X}$, then

\begin{align}\label{final0}
\ell_\lambda=PL_\lambda(I_X+D_0\eta_\lambda)=P L_\lambda\mid_X=L_{\lambda}\mid_X.
\end{align}
Finally, it follows from \eqref{Hamiltonian-finitedimred} that 

\begin{align}\label{final1}
\sfl_G(L)=\sfl_G(\ell),
\end{align}
where $\ell:=\{\ell_{\lambda}\}_{\lambda\in I}$. Hence to prove Theorem \ref{thm-HamiltonianNonlin}, it remains to show the existence of a bifurcation for $\overline{f}$. Note that $\ell_0$ and $\ell_1$ are invertible by \eqref{final0}.\\
Let $\tilde{f}:I\times(B_X\times Y)\rightarrow \mathbb{R}$ be defined by 

\[\tilde{f}(\lambda,x,y)=\overline{f}(\lambda,x)+\frac{1}{2}\|y\|^2_H.\]
$\tilde{f}$ is $G$ invariant as the action of $G$ on $H$ is orthogonal. If $\tilde{L}=\{\tilde{L}_\lambda\}_{\lambda\in I}$ denote the Hessians of $\tilde{f}$, then $\tilde{L}_\lambda\in\mathcal{FS}_+(H)$, $\lambda\in I$, and

\begin{align}\label{final2}
\sfl_G(\tilde{L})=\sfl_G(\ell)\in RO(G)
\end{align}
by Lemma \ref{lemma-splitting}. As $\tilde{L}_0$ and $\tilde{L}_1$ are invertible, and $\sfl_G(\tilde{L})\neq 0$ by \eqref{final1} and \eqref{final2}, we obtain from Corollary \ref{Cor-SmollerWasserman} that there is a bifurcation of critical points for $\tilde{f}$. Finally, $\tilde{f}$ and $\overline{f}$ obviously have the same bifurcation points, which shows Theorem \ref{thm-HamiltonianNonlin}.

 \subsubsection*{Acknowledgements}
     
    The authors were supported by the grant BEETHOVEN2
    of the National Science Centre, Poland, no.\ 2016/23/G/ST1/04081. Moreover, the research was funded by the Deutsche Forschungsgemeinschaft (DFG, German Research Foundation) - 459826435. 
    

\vspace{1cm}
Joanna Janczewska\\
Institute of Applied Mathematics\\
Faculty of Applied Physics and Mathematics\\
Gda\'{n}sk University of Technology\\
Narutowicza 11/12, 80-233 Gda\'{n}sk, Poland\\
joanna.janczewska@pg.edu.pl

\vspace{1cm}
Marek Izydorek\\
Institute of Applied Mathematics\\
Faculty of Applied Physics and Mathematics\\
Gda\'{n}sk University of Technology\\
Narutowicza 11/12, 80-233 Gda\'{n}sk, Poland\\
marek.izydorek@pg.edu.pl

\vspace{1cm}
Nils Waterstraat\\
Institut f\"ur Mathematik\\
Naturwissenschaftliche Fakult\"at II\\
Martin-Luther-Universit\"at Halle-Wittenberg\\
06099 Halle (Saale), Germany\\
nils.waterstraat@mathematik.uni-halle.de

\end{document}